\documentclass[10pt]{amsart}
\usepackage{verbatim}
\usepackage{eucal,url,amssymb,stmaryrd,enumerate,amscd}
\usepackage{amsmath}
\usepackage[pagebackref,colorlinks=true,linkcolor=blue,citecolor=blue]{hyperref}
\usepackage{amsfonts}
\usepackage{amsthm}
\usepackage[margin=1in]{geometry}

\numberwithin{equation}{section}
\newtheorem{thrm}{Theorem}[section]
\newtheorem{lemma}[thrm]{Lemma}
\newtheorem{prop}[thrm]{Proposition}
\newtheorem{cor}[thrm]{Corollary}
\newtheorem{dfn}[thrm]{Definition}

\newtheorem{rmrk}[thrm]{Remark}

\newtheorem{conv}[thrm]{Convention}

 1

\newcommand{\vol}{\, Vol_{\eta}}

\begin{document}

\begin{abstract}
A version of Lichnerowicz' theorem giving a lower bound
of the eigenvalues of the sub-Laplacian under a lower bound on the $%
Sp(1)Sp(1)$ component of the qc-Ricci curvature on a compact seven dimensional quaternionic contact manifold is established. It is shown that in the
case of a seven dimensional compact 3-Sasakian manifold the lower bound is reached if and only if the quaternionic contact manifold is a round 3-Sasakian sphere.
\end{abstract}

\title[The sharp lower bound of the first eigenvalue]{The sharp lower bound
of the first eigenvalue of the sub-Laplacian on a quaternionic
contact manifold in dimension 7}
\date{\today }
\thanks{This work has been partially funded by}
\author{S. Ivanov}
\address[Stefan Ivanov]{University of Sofia, Faculty of Mathematics and
Informatics, blvd. James Bourchier 5, 1164, Sofia, Bulgaria}
\email{ivanovsp@fmi.uni-sofia.bg}
\author{A. Petkov}
\address[Alexander Petkov]{University of Sofia, Faculty of Mathematics and
Informatics, blvd. James Bourchier 5, 1164, Sofia, Bulgaria}
\email{a\_petkov\_fmi@abv.bg}
\author{D. Vassilev}
\address[Dimiter Vassilev]{ Department of Mathematics and Statistics\\
University of New Mexico\\
Albuquerque, New Mexico, 87131-0001}
\email{vassilev@math.unm.edu}
\maketitle
\tableofcontents

\setcounter{tocdepth}{2}

\section{Introduction}

The purpose of this paper is to give a version of the results  of
\cite{IPV1} in the case of a seven dimensional quaternionic
contact   (abbr.   qc) manifold. Thus, the main result is the
establishment of a lower bound for the first eigenvalue on a seven
dimensional qc manifold satisfying a Lichnerowicz type condition
extending the classical \cite{Li} and CR
\cite{Gr,LL,Chi06} Lichnerowicz type results  to the case of a
general compact qc manifold. We also prove an  Obata  type result
\cite{O3} characterizing the extremal case in the qc Lichnerowicz
type result assuming the extremal case is  a  (seven dimensional)
qc-Einstein manifold of constant qc-scalar curvature.

It should be noted that a seven dimensional qc structure,
similarly to a three dimensional  strongly pseudoconvex manifold,
frequently presents some additional difficulties or requires a
different analysis in various geometric and analytic questions,
see for example \cite{IMV1} and \cite{IMV}. Such is the case of
this paper. On the other hand, a notable difference
 are the (conformal) flatness problems,
see \cite{Car,ChM,IVZ} and \cite{IV} respectively, where there is
no distinction between the seven and higher dimensional results in
the qc case.

Another motivation for the current paper comes from a number of
Sasakian  case results   \cite{CC07,CC09b,CC09a}, \cite{Bar} and
\cite{ChW}. These CR results came after Greenleaf \cite{Gr} proved
the CR Lichnerowicz type result on a $2n+1$ dimensional CR
manifold for $n\geq 3$. Subsequently,  \cite{LL} adapted
Greenleaf's proof to cover the case $n=2$, while the case $n=1$
was treated in \cite{LL} assuming further a condition on the
covariant derivative of the pseudohermitian torsion. In
\cite{Chi06} the author also considered  the $n=1$ case replacing
the additional assumption with the \textit{CR-invariant condition}
that the CR-Paneitz operator is non-negative. As well known, for
$n>1$ the CR-Paneitz operator is always non-negative,  while in
the case $n=1$ the vanishing of the pseudohermitian torsion
implies that the CR-Paneitz operator is non-negative,  see
\cite{Chi06} and\cite{CCC07}.

In the current paper we introduce a certain $P-$function in the
setting of a qc structure which allows us to obtain   in
dimension seven the qc Lichnerowicz - Obata type results, except for the
latter we require that the extremal qc structure is  3-Sasakian.
This  condition means that $M$ is a qc-Einstein structure
of constant  positive qc-scalar curvature, see \cite{IMV}, and also \cite{IV1} and
\cite{IV2} for the negative scalar curvature case. Note that in  higher dimensions the qc-Einstein condition implies the constancy of the qc-scalar curvature \cite{IMV}.   In view of
\cite{IVO}, \cite{LW1} and \cite{IV3}, where the Obata type result
was proven for a general CR manifold, our qc Obata-type result is
not completely satisfactory, but the non qc-Einstein case seems to
present extra difficulties. In addition, the introduced $P-$form
is interesting on its own.  Here $P$ should remind not
only of the Paneitz operator, but also of the $P-$function used in
the theory of elliptic partial differential equations, see
\cite{Sperb} for references and some results.

\begin{thrm}
\label{mainpan} Let $(M,g,\mathbb{Q})$ be a compact quaternionic contact
manifold of dimension seven. Suppose there is a positive constant $k_0$
such that the normalized scalar curvature $S$ and the torsion $T^0$
satisfy the Lichnerowicz type inequality
\begin{equation}  \label{7Dcondm-app}
6Sg(X,X)+10T^0(X,X)\geq k_0 g(X,X).
\end{equation}
If, in addition, the  $P-$function of any eigenfunction of the sub-Laplacian is non-negative,
then for any eigenvalue $\lambda$ of the sub-Laplacian $\triangle$ we have  the inequality
\begin{equation*}
\lambda \ge \frac13k_0.
\end{equation*}
\end{thrm}
The $P$-function of a smooth function $f$, cf. Definition~\ref{d:def P}, is defined with the help of the Biquard connection, the qc-scalar curvature and the $[-1]$ component of the torsion tensor, see Section \ref{ss:Biq conn}, \eqref{qscs} and \eqref{Tcompnts}.  We say that the $P-$function  of $f$ is non-negative if
\begin{equation*}  
-\int_M P_f\vol\geq 0.
\end{equation*}

We note that  the Lichnerowicz type assumption in any dimension is
\begin{equation}  \label{condm-app}
2(n+2)Sg(X,X)+\frac{4n^2+14n+12}{2n+1}T^0(X,X)+\frac{4(n+2)^2(2n-1)}{%
(n-1)(2n+1)}U(X,X)\geq k_0 g(X,X).
\end{equation}
When $n>1$ as shown in \cite{IPV1} it implies that for any eigenvalue $\lambda$ of the sub-Laplacian $\triangle$ we have  the lower bound
\begin{equation*}
\lambda \ge \frac{n}{n+2}k_0.
\end{equation*}
In fact,  both the seven and higher dimensional cases follow from
the positivity of the $P$-function.

A stronger condition than the one required in Theorem
\ref{mainpan} is the assumption that the $P$-function of any
smooth function is non-negative  in which case we say that the
$C-$operator of $M$ is non-negative. Here $C$ is a fourth-order
differential operator  on $M$ (independent of $f$!) defined by
\[
Cf =-\nabla^* P_f=(\nabla_{e_a} P_f)\,(e_a).
\]
It turns out that this stronger condition holds on any  compact qc
manifold of dimension at least eleven, see Theorem~\ref{pan}. On
the other hand we shall prove in Proposition~\ref{d3n1} that on a
seven dimensional qc-Einstein manifold of constant non-negative
scalar curvature the $P-$function of any eigenfunction of the
sub-Laplacian is non-negative. This fact implies the following
\begin{cor}\label{c:LOin7D}
Let $(M,g,\mathbb{Q})$ be a compact quaternionic contact manifold
of dimension seven which is qc-Einstein and of constant
normalized positive scalar curvature $S=2$. Then for any
eigenvalue $\lambda$ of the sub-Laplacian $\triangle$ we have  the
inequality
\begin{equation}\label{71}
\lambda \ge 4.
\end{equation}
Furthermore, $\lambda=4$ is an eigenvalue of the sub-Laplacian if and only if $M$ is the unit seven dimensional 3-Sasakian sphere.

In particular, on a seven dimensional 3-Sasakian manifold any
eiqnfunction $\lambda$ of the sub-Laplacian satisfies the
inequality \eqref{71} and the equality is attained only on the
seven dimensional 3-Sasakian sphere.
\end{cor}
We note that in \cite{IMV2} is given an explicit formula for the
eigenfunctions of the above eigenvalue, see also \cite{ACB}.

\begin{conv}
\label{conven} \hfill\break\vspace{-15pt}
\begin{enumerate}
\item[a)] We shall use $X,Y,Z,U$ to denote horizontal vector fields, i.e. $%
X,Y,Z,U\in H$.

\item[b)] $\{e_1,\dots,e_{4n}\}$ denotes a local orthonormal basis of the
horizontal space $H$.

\item[c)] The summation convention over repeated vectors from the basis $%
\{e_1,\dots,e_{4n}\}$ will be used. For example, for a (0,4)-tensor $P$, the
formula $k=P(e_b,e_a,e_a,e_b)$ means
\begin{equation*}
k=\sum_{a,b=1}^{4n}P(e_b,e_a,e_a,e_b);
\end{equation*}

\item[d)] The triple $(i,j,k)$ denotes any cyclic permutation of $(1,2,3)$.

\item[e)] $s$ will be any number from the set $\{1,2,3\}, \quad
s\in\{1,2,3\} $.
\end{enumerate}
\end{conv}

\section{Quaternionic contact manifolds}

Quaternionic contact manifolds were introduced in \cite{Biq1}.  We also refer to \cite{IMV} and \cite{IV} for results and further background.

\subsection{Quaternionic contact structures and the Biquard connection}\label{ss:Biq conn}

A quaternionic contact (qc) manifold $(M, g, \mathbb{Q})$ is a $4n+3$%
-dimensional manifold $M$ with a codimension three distribution $H$ equipped with an $Sp(n)Sp(1)$ structure. Explicitly, $H$ is the kernel of a local 1-form $\eta=(\eta_1,\eta_2,\eta_3)$ with values in $\mathbb{R}^3$ together with a compatible Riemannian metric $g$ and a rank-three bundle $\mathbb{Q}$ consisting of endomorphisms of $H$ locally generated by three almost complex structures $I_1,I_2,I_3$ on $H$ satisfying the identities of the imaginary
unit quaternions. Thus, we have $I_1I_2=-I_2I_1=I_3, \quad I_1I_2I_3=-id_{|_H}$ which are hermitian compatible with the metric $g(I_s.,I_s.)=g(.,.)$ and the following
compatibility condition holds $$ 2g(I_sX,Y)\ =\ d\eta_s(X,Y), \quad
X,Y\in H.$$
On a qc manifold of dimension $(4n+3)>7$ with a fixed metric $g$ on the horizontal distribution $H$ there exists a canonical
connection defined in \cite{Biq1}.  In fact, Biquard showed that there  is a unique connection $\nabla$ with torsion $T$  and a unique supplementary subspace $V$ to $H$ in $TM$, such that:
\begin{enumerate}[(i)]
\item $\nabla$ preserves the decomposition $H\oplus V$ and the $%
Sp(n)Sp(1)$ structure on $H$, i.e. $\nabla g=0, \nabla\sigma \in\Gamma(%
\mathbb{Q})$ for a section $\sigma\in\Gamma(\mathbb{Q})$, and its torsion on
$H$ is given by $T(X,Y)=-[X,Y]_{|V}$;
\item for $\xi\in V$, the endomorphism $T(\xi,.)_{|H}$ of $H$ lies in $%
(sp(n)\oplus sp(1))^{\bot}\subset gl(4n)$;
\item the connection on $V$ is induced by the natural identification $%
\varphi$ of $V$ with the subspace $sp(1)$ of the endomorphisms of $H$, i.e. $%
\nabla\varphi=0$.
\end{enumerate}

In ii), the inner product $<,>$ of $End(H)$ is given by $<A,B> = {\
\sum_{i=1}^{4n} g(A(e_i),B(e_i)),}$ for $A, B \in End(H)$. We shall call the
above connection \emph{the Biquard connection}. When the dimension of $M$ is
at least eleven \cite{Biq1} also described the supplementary distribution $V$%
, which is (locally) generated by the so called Reeb vector fields $%
\{\xi_1,\xi_2,\xi_3\}$ determined by
\begin{equation}  \label{bi1}
\begin{aligned} \eta_s(\xi_k)=\delta_{sk}, \qquad (\xi_s\lrcorner
d\eta_s)_{|H}=0,\\ (\xi_s\lrcorner d\eta_k)_{|H}=-(\xi_k\lrcorner
d\eta_s)_{|H}, \end{aligned}
\end{equation}
where $\lrcorner$ denotes the interior multiplication.

If the dimension of $%
M $ is seven Duchemin shows in \cite{D} that if we assume, in addition, the
existence of Reeb vector fields as in \eqref{bi1}, then the Biquard result holds. \emph{Henceforth, by a qc structure in dimension $7$ we shall mean a qc
structure satisfying \eqref{bi1}}. This implies the existence of the connection with properties (i), (ii) and (iii) above.

The fundamental 2-forms $\omega_s$ of the quaternionic structure $Q$ are
defined by
\begin{equation}  \label{thirteen}
2\omega_{s|H}\ =\ \, d\eta_{s|H},\qquad \xi\lrcorner\omega_s=0,\quad \xi\in
V,
\end{equation}
and the torsion restricted to $H$ has the form
\begin{equation}  \label{torha}
T(X,Y)=-[X,Y]_{|V}=2\omega_1(X,Y)\xi_1+2\omega_2(X,Y)\xi_2+2\omega_3(X,Y)%
\xi_3.
\end{equation}

\subsection{Invariant decompositions}

Any endomorphism $\Psi$ of $H$ can be decomposed with respect to the
quaternionic structure $(\mathbb{Q},g)$ uniquely into four $Sp(n)$-invariant
parts 
$\Psi=\Psi^{+++}+\Psi^{+--}+\Psi^{-+-}+\Psi^{--+},$ 
where $\Psi^{+++}$ commutes with all three $I_i$, $\Psi^{+--}$ commutes with
$I_1$ and anti-commutes with the others two and etc. The two $Sp(n)Sp(1)$-invariant components
\index{$Sp(n)Sp(1)$-invariant components!$\Psi_{[3]}$}
\index{$Sp(n)Sp(1)$-invariant components!$\Psi_{[-1]}$} are given by
\begin{equation*}
\Psi_{[3]}=\Psi^{+++},\quad \Psi_{[-1]}=\Psi^{+--}+\Psi^{-+-}+\Psi^{--+}.
\end{equation*}
They are the projections on the eigenspaces of the Casimir operator
\begin{equation}  \label{e:cross}
\Upsilon =\ I_1\otimes I_1\ +\ I_2\otimes I_2\ +\ I_3\otimes I_3,
\end{equation}
corresponding, respectively, to the eigenvalues $3$ and $-1$, see \cite{CSal}. Note here that each of the three 2-forms $%
\omega_s$ belongs to its [-1]-component, $\omega_s=\omega_{s[-1]}$ and
constitute a basis of the lie algebra $sp(1)$. In particular, using that $\left\{ \frac{1}{2\sqrt{n}}\omega _{s}\right\} $ is an
orthonormal set in $\Psi _{\lbrack -1]}$ we have
\begin{equation}
|(\nabla ^{2}f)_{[-1]}|^{2}\geq \frac{1}{4n}\sum_{s=1}^{3}\left[ g\left(
\nabla ^{2}f,\omega _{s}\right) \right] ^{2}  \label{coshy1}
\end{equation}%
while a projection on $\left\{ \frac{1}{2\sqrt{n}}g\right\} $ gives
\begin{equation}
|(\nabla ^{2}f)_{[3]}|^{2}\geq \frac{1}{4n}(\triangle f)^{2}.  \label{coshy3}
\end{equation}%
If $n=1$ then the space of symmetric endomorphisms commuting with all $I_s$ is 1-dimensional, i.e., the [3]-component of any symmetric endomorphism $\Psi$ on $H$ is proportional to the identity, $\Psi_{3}=-\frac{tr\Psi}{4}Id_{|H}$.  Thus, when $n=1$ we have the following fact.
\begin{lemma}\label{l:inv_decomp_7D}
The space $\Psi _{\lbrack 3]}$ is four dimensional and the symmetric
tensors in it are proportional to the metric. The space $\Psi _{\lbrack -1]}$
is twelve dimensional, in which lies the three dimensional space of the
2-forms $\omega _{i}$. The latter determines the anti-symmetric part of the $%
\Psi _{\lbrack -1]}$ component.
\end{lemma}

\subsection{The torsion tensor}

The torsion endomorphism $T_{\xi }=T(\xi ,\cdot ):H\rightarrow H,\quad \xi
\in V$ will be decomposed into its symmetric part $T_{\xi }^{0}$ and skew-symmetric part $b_{\xi
},T_{\xi }=T_{\xi }^{0}+b_{\xi }$. Biquard showed \cite{Biq1} that the
torsion $T_{\xi }$ is completely trace-free, $tr\,T_{\xi }=tr\,T_{\xi }\circ
I_{s}=0$, its symmetric part has the properties $T_{\xi
_{i}}^{0}I_{i}=-I_{i}T_{\xi _{i}}^{0}\quad I_{2}(T_{\xi
_{2}}^{0})^{+--}=I_{1}(T_{\xi _{1}}^{0})^{-+-},\quad I_{3}(T_{\xi
_{3}}^{0})^{-+-}=I_{2}(T_{\xi _{2}}^{0})^{--+},\quad I_{1}(T_{\xi
_{1}}^{0})^{--+}=I_{3}(T_{\xi _{3}}^{0})^{+--}$, where the upperscript $+++$
means commuting with all three $I_{i}$, $+--$ indicates commuting with $%
I_{1} $ and anti-commuting with the other two and etc. The skew-symmetric
part can be represented as $b_{\xi _{i}}=I_{i}U$, where $U$ is a traceless
symmetric (1,1)-tensor on $H$ which commutes with $I_{1},I_{2},I_{3}$.
Therefore we have $T_{\xi _{i}}=T_{\xi _{i}}^{0}+I_{i}U$.  If $n=1$ then the tensor $U$ vanishes identically, $U=0$, and the torsion is
a symmetric tensor, $$T_{\xi }=T_{\xi }^{0}.$$  The two $Sp(n)Sp(1)$-invariant trace-free symmetric 2-tensors on $H$
\begin{equation}\label{Tcompnts}
T^0(X,Y)=
g((T_{\xi_1}^{0}I_1+T_{\xi_2}^{0}I_2+T_{ \xi_3}^{0}I_3)X,Y) \ \text{ and }\  U(X,Y)
=g(uX,Y)
\end{equation}
were introduced in \cite{IMV}  and enjoy the properties
\begin{equation}  \label{propt}
\begin{aligned} T^0(X,Y)+T^0(I_1X,I_1Y)+T^0(I_2X,I_2Y)+T^0(I_3X,I_3Y)=0, \\
U(X,Y)=U(I_1X,I_1Y)=U(I_2X,I_2Y)=U(I_3X,I_3Y). \end{aligned}
\end{equation}
From \cite[Proposition~2.3]{IV} we have
\begin{equation}  \label{need}
4T^0(\xi_s,I_sX,Y)=T^0(X,Y)-T^0(I_sX,I_sY),
\end{equation}
hence, taking into account \eqref{need} it follows
\begin{equation}  \label{need1}
T(\xi_s,I_sX,Y)=T^0(\xi_s,I_sX,Y)+g(I_suI_sX,Y)=\frac14\Big[%
T^0(X,Y)-T^0(I_sX,I_sY)\Big]-U(X,Y).
\end{equation}

Any 3-Sasakian manifold has zero torsion endomorphism, and the converse is
true if in addition the qc scalar curvature (see \eqref{qscs}) is a positive
constant \cite{IMV}.

\subsection{Torsion and curvature}

Let $R=[\nabla,\nabla]-\nabla_{[\ ,\ ]}$ be the curvature tensor of $\nabla$
and the dimension is $4n+3$. We denote the curvature tensor of type (0,4)
and the torsion tensor of type (0,3) by the same letter, $%
R(A,B,C,D):=g(R(A,B)C,D),\quad T(A,B,C):=g(T(A,B),C)$, $A,B,C,D \in
\Gamma(TM)$. The Ricci tensor, the normalized scalar curvature and the Ricci
$2$-forms of the Biquard connection, called \emph{qc-Ricci tensor} $Ric$,
\emph{normalized qc-scalar curvature} $S$ and \emph{qc-Ricci forms} $\rho_s,
\tau_s$, respectively, are defined by
\begin{equation}  \label{qscs}
\begin{aligned} & Ric(A,B)=R(e_b,A,B,e_b),\quad
8n(n+2)S=R(e_b,e_a,e_a,e_b),\quad \rho_s(A,B)=\frac1{4n}R(A,B,e_a,I_se_a),
\\
& \tau_s(A,B)=\frac1{4n}R(e_a,I_se_a,A,B,),\quad
\zeta_s(A,B)=\frac1{4n}R(e_a,A,B,I_se_a). \end{aligned}
\end{equation}

\begin{dfn}
A qc structure is said to be qc Einstein if the horizontal qc-Ricci tensor
is a scalar multiple of the metric,
\begin{equation*}
Ric(X,Y)=2(n+2)Sg(X,Y).
\end{equation*}
\end{dfn}
As shown in \cite{IMV}, see
also \cite{IMV1,IV,IPV1} for a different proof, the qc Einstein condition is equivalent to the vanishing of the torsion
endomorphism of the Biquard connection. In this case $S$ is constant and the
vertical distribution is integrable provided $n>1$. It is also worth recalling that the horizontal qc-Ricci tensors and the integrability of the vertical distribution can be
expressed in terms of the torsion of the Biquard connection \cite{IMV} (see
also \cite{IMV1,IV,IPV1}). For example, we have

\begin{equation}  \label{sixtyfour}
\begin{aligned} & Ric(X,Y) =(2n+2)T^0(X,Y)+(4n+10)U(X,Y)+2(n+2)Sg(X,Y),\\ &
\zeta_s(X,I_sY)=\frac{2n+1}{4n}T^0(X,Y)+\frac1{4n}T^0(I_sX,I_sY)+%
\frac{2n+1}{2n}U(X,Y)+\frac{S}2g(X,Y), \\ & T(\xi_{i},\xi_{j})
=-S\xi_{k}-[\xi_{i},\xi_{j}]_{H}, \qquad S = -g(T(\xi_1,\xi_2),\xi_3),\\ &
g(T(\xi_i,\xi_j),X)
=-\rho_k(I_iX,\xi_i)=-\rho_k(I_jX,\xi_j)=-g([\xi_i,\xi_j],X). \end{aligned}
\end{equation}
Note that for $n=1$ the above formulas hold with $U=0$.

We shall also need the general formula for the curvature \cite{IV,IV2}
\begin{multline}
R(\xi _{i},X,Y,Z)=-(\nabla _{X}U)(I_{i}Y,Z)+\omega _{j}(X,Y)\rho
_{k}(I_{i}Z,\xi _{i})-\omega _{k}(X,Y)\rho _{j}(I_{i}Z,\xi _{i})
\label{d3n5} \\
-\frac{1}{4}\Big[(\nabla _{Y}T^{0})(I_{i}Z,X)+(\nabla _{Y}T^{0})(Z,I_{i}X)%
\Big]+\frac{1}{4}\Big[(\nabla _{Z}T^{0})(I_{i}Y,X)+(\nabla
_{Z}T^{0})(Y,I_{i}X)\Big] \\
-\omega _{j}(X,Z)\rho _{k}(I_{i}Y,\xi _{i})+\omega _{k}(X,Z)\rho
_{j}(I_{i}Y,\xi _{i})-\omega _{j}(Y,Z)\rho _{k}(I_{i}X,\xi _{i})+\omega
_{k}(Y,Z)\rho _{j}(I_{i}X,\xi _{i}),
\end{multline}%
where the Ricci two forms are given by, cf. \cite[Theorem 3.1]{IV} or \cite[Theorem4.3.11]{IV2}
\begin{equation}
\begin{aligned}
6(2n+1)\rho_s(\xi_s,X)=(2n+1)X(S)+\frac12(%
\nabla_{e_a}T^0)[(e_a,X)-3(I_se_a,I_sX)]-2(\nabla_{e_a}U)(e_a,X),\\
6(2n+1)\rho_i(\xi_j,I_kX)=(2n-1)(2n+1)X(S)-\frac12(%
\nabla_{e_a}T^0)[(4n+1)(e_a,X)+3(I_ie_a,I_iX)]\\-4(n+1)(%
\nabla_{e_a}U)(e_a,X) .\end{aligned}  \label{d3n6}
\end{equation}

\subsection{The Ricci identities}
We shall use repeatedly the following Ricci identities of order two and
three, see also \cite{IV}. 
Let $\xi_i$, $i=1,2,3$ be the Reeb vector fields, $X, Y\in H$ and $f$ a
smooth function on the qc manifold $M$ with $\nabla f$ its horizontal
gradient of $f$, $g(\nabla f,X)=df(X)$. We have the following Ricci identities
\begin{equation}  \label{boh2}
\begin{aligned}
& \nabla^2f
(X,Y)-\nabla^2f(Y,X)=-2\sum_{s=1}^3\omega_s(X,Y)df(\xi_s)\\ & \nabla^2f
(X,\xi_s)-\nabla^2f(\xi_s,X)=T(\xi_s,X,\nabla f)\\
& \nabla^3 f
(X,Y,Z)-\nabla^3 f(Y,X,Z)=-R(X,Y,Z,\nabla f) - 2\sum_{s=1}^3
\omega_s(X,Y)\nabla^2f (\xi_s,Z)\\
& \nabla ^{3}f(\xi _{i},X,Y) =\nabla ^{3}f(X,Y,\xi _{i})-\nabla ^{2}f\left(
T\left( \xi _{i},X\right) ,Y\right) -\nabla ^{2}f\left( X,T\left( \xi
_{i},Y\right) \right) -df\left( \left( \nabla _{X}T\right) \left( \xi
_{i},Y\right) \right)\\ &\hskip4in -R(\xi _{i},X,Y,\nabla f) .
\end{aligned}
\end{equation}
 In particular we have
\begin{equation}  \label{xi1}
g(\nabla^2f , \omega_s) =\nabla^2f(e_a,I_se_a)=-4ndf(\xi_s).
\end{equation}
Now, Lemma \ref{l:inv_decomp_7D} and the Ricci identities show that for any smooth function $f $ on a seven dimensional qc manifold we have
\begin{equation}\label{hess7d}
\begin{aligned}
& (\nabla ^{2}f)_{[3]}(X,Y)=-\frac{\triangle f}{4}g(X,Y),\qquad (\nabla
^{2}f)_{[3][a]}(X,Y)=\frac{1}{2}\left[ (\nabla ^{2}f)_{[3]}(X,Y)-(\nabla
^{2}f)_{[3]}(Y,X)\right] =0 \\
& (\nabla ^{2}f)_{[-1][a]}(X,Y)=-df(\xi _{i})\omega _{i}(X,Y).
\end{aligned}
\end{equation}

\subsection{The horizontal divergence theorem and qc-normal frames}

Let $(M, g,\mathbb{Q})$ be a qc manifold of dimension $4n+3\geq 7$ For a fixed
local 1-form $\eta$ and a fix $s\in \{1,2,3\}$ the form
\begin{equation}  \label{e:volumeform}
Vol_{\eta}=\eta_1\wedge\eta_2\wedge\eta_3\wedge\omega_s^{2n}
\end{equation}
is a locally defined volume form. Note that $Vol_{\eta}$ is independent on $%
s $ as well as it is independent on the local one forms $\eta_1,\eta_2,%
\eta_3 $. Hence it is globally defined volume form denoted with $\,
Vol_{\eta}$.

We define the (horizontal) divergence of a horizontal vector
field/one-form $\sigma\in\Lambda^1\, (H)$ as
\begin{equation}  \label{e:divergence}
\nabla^*\, \sigma\ =-tr|_{H}\nabla\sigma=\ -\nabla \sigma(e_a,e_a).
\end{equation}
The integration by parts formula from \cite{IMV}, see also \cite{Wei}, takes the form
\begin{equation}\label{div}
\int_M (\nabla^*\sigma)\,\vol\ =\ 0.
\end{equation}

Finally, we recall that an orthonormal frame\newline
\centerline{$\{e_1,e_2=I_1e_1,e_3=I_2e_1,e_4=I_3e_1,\dots,
e_{4n}=I_3e_{4n-3}, \xi_1, \xi_2, \xi_3 \}$}\newline
is a qc-normal frame (at a point) if the connection 1-forms of the Biquard
connection vanish (at that point). Lemma~4.5 in \cite{IMV} asserts that a
qc-normal frame exists at each point of a qc manifold.

\section{The $P-$form}
Let $(M,g,\mathbb{Q})$ be a compact quaternionic contact
manifold of dimension $4n+3$ and $f$ a smooth function on $M$.

\begin{dfn}
\label{d:def P} For a fixed $f$ we define  a  one form $P\equiv P_f \equiv P[f]$ on $M$, which we call the $P-$form of $f$, by the following equation
\begin{equation}
\begin{aligned} P_f(X) =&\nabla ^{3}f(X,e_{b},e_{b})+\sum_{t=1}^{3}\nabla
^{3}f(I_{t}X,e_{b},I_{t}e_{b})-4nSdf(X)+4nT^{0}(X,\nabla
f)-\frac{8n(n-2)}{n-1}U(X,\nabla f),\\ &\hskip4.5in \text{if $n>1$},\\
P_f(X) =&\nabla
^{3}f(X,e_{b},e_{b})+\sum_{t=1}^{3}\nabla
^{3}f(I_{t}X,e_{b},I_{t}e_{b})-4Sdf(X)+4T^{0}(X,\nabla f),\quad\text{if $n=1$}. \end{aligned}
\label{e:def P}
\end{equation}%
The $P-$function of  $f$ is the function $P_f(\nabla f)$. Finally, the $C-$operator is the fourth-order differential operator  on $M$ (independent of $f$!) defined by
\begin{equation}  \label{e:def C}
Cf =-\nabla^* P_f=(\nabla_{e_a} P_f)\,(e_a).
\end{equation}
We say that the $P-$function of  $f$ is non-negative  if its integral exists and is non-positive
\begin{equation}  \label{e:non-negative Paneitz}
\int_M f\cdot Cf \vol= -\int_M
P_f(\nabla f)\vol\geq 0.
\end{equation}
If \eqref{e:non-negative Paneitz} holds for any smooth function of compact support we say that the $C-$operator is non-negative.
\end{dfn}

One of the key identities which relates  the P-function and the Bochner formula is given in the following
\begin{lemma}\label{gr3} On a qc manifold of dimension $4n+3$ we have
\begin{equation}  \label{e:gr3}
\sum_{s=1}^3\nabla^{2}f(\xi_s,I_sX)=\frac{1}{4n}\sum_{s=1}^3%
\nabla^{3}f(I_sX,I_se_a,e_a)-\sum_{s=1}^3T(\xi_s,I_sX,\nabla f).
\end{equation}
In addition, if the manifold is compact, then the next integral formula
holds
\begin{equation}  \label{e:gr4}
\int_M\sum_{s=1}^3\nabla^2f(\xi_s,I_s\nabla f)\vol=\int_M\Big[-%
\frac{1}{4n}P_n(\nabla f)-\frac{1}{4n}(\triangle f)^2-S|\nabla f|^2+\frac{%
(n+1)}{n-1}U(\nabla f,\nabla f)\Big]\vol.
\end{equation}
\end{lemma}

\begin{proof}
Working in a qc-normal frame, taking into account the $Sp(n)Sp(1)$ invariance of $\nabla^2f(\xi_s,I_s\nabla f)$, we have
\begin{eqnarray*}
\sum_{s=1}^3\nabla^3f(I_sX,I_se_a,e_a)=\sum_{s=1}^3\nabla_{I_sX}[%
\nabla^2f(I_se_a,e_a)]=\sum_{s=1}^3\nabla_{I_sX}[4ndf(\xi_s)]=4n\sum_{s=1}^3%
\nabla^2f(I_sX,\xi_s) \\
=4n\sum_{s=1}^3\nabla^2f(\xi_s,I_sX)+4n\sum_{s=1}^3T(\xi_s,I_sX,\nabla f),
\end{eqnarray*}
where we used \eqref{xi1} for the second equality and the second equality in %
\eqref{boh2} for the fourth one.

For the second part of the lemma, first we express the term $%
\sum_{s=1}^3\nabla^{3}f(I_sX,I_se_a,e_a)$ in \eqref{e:gr3} by the definition %
\eqref{e:def P} of $P_f$. Then we replace $X$ by $\nabla f$, integrate over $%
M$ and use the easily checked equality
\begin{equation}  \label{lap}
\int_M\nabla^3(\nabla f,e_a,e_a)\vol=-\int_M(\triangle
f)^2\vol
\end{equation}
for the term $\nabla^3f(\nabla f,e_a,e_a)$. Finally  we use
\begin{equation}  \label{boh4}
2\sum_{s=1}^3T(\xi_s,I_s\nabla f,\nabla f)=2T^0(\nabla f,\nabla f)-6U(\nabla
f,\nabla f),
\end{equation}
which follows from \eqref{need1} together with \eqref{propt}.
This proves \eqref{e:gr4}.
\end{proof}

\subsection{The non-negativity of the $C$-operator for $n>1$}

Let $B_0$ be the trace-free part of the 3-component of the horizontal Hessian.
\begin{equation}  \label{np1}
4B_0(X,Y)=\nabla^2f(X,Y)+\nabla^2f(I_1X,I_1Y)+\nabla^2f(I_2X,I_2Y)+%
\nabla^2f(I_3X,I_3Y)+\frac1n\triangle fg(X,Y).
\end{equation}

\begin{thrm}
\label{pan} On a qc manifold of dimension $4n+3$ we have the formula
\begin{equation}  \label{panon}
4(\nabla_{e_a}B_0)(e_a,X)=\frac{n-1}nP_n(X).
\end{equation}
In particular, if the manifold is compact then the $C-$operator is
non-negative for any dimension bigger than seven. In this case for any function $f$ the function $Cf$ vanishes exactly when the trace-free part of the 3-component of a
function vanishes. In this case the $P-$form of $f$ vanishes as well.
\end{thrm}

\begin{proof}
We have from the Ricci identities that:
\begin{equation}  \label{pan111}
\nabla^3f(e_a,e_a,X)=\nabla^3f(X,e_a,e_a)+Ric(X,\nabla
f)+4\sum_{s=1}^3\nabla^2(\xi_s,I_sX)+2\sum_{s=1}^3T(\xi_s,I_sX,\nabla f)
\end{equation}
Similarly, we obtain
\begin{multline}  \label{pan112}
\nabla^3f(e_a,I_se_a,I_sX)=\nabla^3f(I_sX,e_a,I_se_a)+4n\zeta_s(I_sX,%
\nabla f)-2\sum_{t=1}^3\omega_t(I_se_a,I_sX)T(\xi_t,e_a,\nabla f) \\
= \nabla^3f(I_sX,e_a,I_se_a)+4n\zeta_s(I_sX,\nabla
f)-2\sum_{s=1}^3T(\xi_s,I_sX,\nabla f).
\end{multline}
A substitution of \eqref{pan111}, \eqref{pan112}, \eqref{e:gr3} and the first and
second equality of \eqref{sixtyfour} in \eqref{np1} imply \eqref{panon}.  The last statement follows by integration by parts using the orthogonality of the components of the horizontal Hessian.
\end{proof}

\subsection{Non-negativity of the $P-$functions of eigenfunctions in dimension seven}

In dimension seven we show the following

\begin{prop}
\label{d3n1} On a seven dimensional compact qc-Einstein manifold of
constant positive qc scalar curvature, $S=2$, the $P-$function of any eigenfunction of the sub-Laplacian is non-negative.
\end{prop}

\begin{proof}
Suppose $f$ satisfies $\triangle f=-\nabla^2f(e_a,e_a)=\lambda f$. We
calculate
\begin{multline}  \label{d3n2}
\int_M|P_f|^2\vol=\int_M\Big[ \nabla
^{3}f(e_a,e_{b},e_{b})+\sum_{t=1}^{3}\nabla
^{3}f(I_{t}e_a,e_{b},I_{t}e_{b})-4Sdf(e_a)\Big]^2\vol \\
=-4S\int_MP_1(\nabla f)\vol-\int_M\nabla^2f(e_b,e_b)Cf\vol
\\
-\int_M\sum_{s,t=1}^3\nabla^2f(e_b,I_se_b)\Big[\nabla^4f(I_se_a,e_a,e_c,e_c)+%
\nabla^4f(I_se_a,I_te_a,e_c,I_te_c)-4S\nabla^2f(I_se_a,e_a)\Big]\vol \\
=-(\lambda+4S)\int_MP_1(\nabla f)\,
Vol_{\eta}-4S\int_M\sum_{s=1}^3[\nabla^2f(e_a,I_se_a)]^2\vol \\
-
8\int_M\sum_{s=1}^3\nabla^2f(e_b,I_se_b)[\nabla^2(\xi_s,e_c,e_c)-%
\nabla^2(e_c,e_c,\xi_s)]\vol
-\int_M\sum_{s\not=t}^3\nabla^2f(e_b,I_se_b)\nabla^4f(I_se_a,I_te_a,e_c,I_te_c)
\\
= -(\lambda+4S)\int_MP_f(\nabla f)\,
Vol_{\eta}-8^3\int_M\sum_{s=1}^3[df(\xi_s)]^2\,
Vol_{\eta}+8\int_Mdf(\xi_i)8^2[\nabla^2(\xi_k,\xi_j)-\nabla^2(\xi_j,\xi_k)]%
\vol,
\end{multline}
where we used the first Ricci identity in \eqref{boh2} several times and the
fact that third covariant derivatives commute on a qc-Einstein manifold with
constant scalar curvature because of  \eqref{d3n6} and the last of the Ricci identities in \eqref{boh2}. The assumptions $T^0=U=0, S=2=const$ together with the second equality in \eqref{d3n6} and the fourth equality in \eqref{sixtyfour} imply that the vertical space is integrable, $g([\xi_s,\xi_t],X)=0$. Now, the Ricci identity,  the integrability of the vertical space  and the third equality in \eqref{sixtyfour} yield
\begin{equation*}
\nabla^2f(\xi_k,\xi_j)-\nabla^2f(\xi_j,\xi_k)=T(\xi_j,\xi_k,df)=-2df(\xi_i).
\end{equation*}
A substitution of this equality in \eqref{d3n2} gives
\begin{multline}  \label{d3n3}
\int_M|P_f|^2\vol= -(\lambda+8)\int_MP_f(\nabla f)\,
Vol_{\eta}-8^3\int_M\sum_{s=1}^3[df(\xi_s)]^2\,
Vol_{\eta}-2\times8^3\int_M\sum_{s=1}^3[df(\xi_s)]^2\vol
\end{multline}
and the non-negativity of the $P-$function of $f$ follows since $\lambda>0$.
\end{proof}

\section{The Lichnerowicz type result for a seven dimensional qc structure}
Here we shall prove our main Theorem \ref{mainpan}. From  \eqref{hess7d} and \eqref{xi1}  we have
\begin{equation}\label{xiii}
|(\nabla ^{2}f)_{[3]}|^{2}=\frac{\left( \triangle f\right) ^{2}}{4},\qquad
|(\nabla ^{2}f)_{[-1][a]}|^{2}=4\sum_{s=1}^{3}(\xi _{s}f)^{2}.
\end{equation}
Next, we recall the four identities given  by   \cite[Lemma 3.3]{IPV1}, \cite[Lemma 3.4]{IPV1}, Lemma \ref{gr3}, and the Bochner identity \cite[(4.1)]{IPV1}.  However, in this lowest dimension, Bochner's identity  \cite[(4.1)]{IPV1} and \cite[Lemma 3.3]{IPV1} are identical. Therefore, in dimension seven, using $U=0$, \eqref{xiii} and \eqref{boh4}  we have the following three identities
\begin{equation}  \label{e:3eqs in dim 7}
\begin{aligned} & \int_{M}\sum_{s=1}^{3}\nabla ^{2}f(\xi _{s},I_{s}\nabla
f)\,Vol_{\eta }=-\int_{M}\Big[|(\nabla ^{2}f)_{[-1][a]}|^{2}+T^{0}(\nabla
f,\nabla f)\Big]\,Vol_{\eta }, \\ & \int_{M}\sum_{s=1}^{3}\nabla ^{2}f(\xi
_{s},I_{s}\nabla f)\,Vol_{\eta }=\int_{M}\Big[\frac{3}{16}(\triangle
f)^{2}-\frac{1}{4}|(\nabla ^{2}f)_{[-1][a]}|^{2}-\frac{1}{4}|(\nabla
^{2}f)_{[-1][s]}|^{2}\Big]\vol\\ & \hskip3.3in
+\int_{M}\Big[-\frac{3}{2}T^{0}(\nabla f,\nabla f)-\frac{3}{2}S|\nabla
f|^{2}\Big]\,\,Vol_{\eta }, \\ & \int_{M}\sum_{t=1}^{3}\nabla ^{2}f(\xi
_{t},I_{t}\nabla f)\,Vol_{\eta }=\int_{M}\Big[ -T^{0}(\nabla f,\nabla
f)-\frac{1}{4}\sum_{t=1}^{3}\nabla ^{3}f(I_{t}\nabla
f,e_{b},I_{t}e_{b})\Big]\,Vol_{\eta }, \end{aligned}
\end{equation}\
corresponding to \cite[Lemma 3.4]{IPV1},  \cite[Lemma 3.3]{IPV1} and \eqref{e:gr3}, respectively.

The a-priori Lichnerowicz type assumption gives the inequality
\begin{equation*}
10T^{0}(\nabla f,\nabla f)+6S|\nabla f|^{2}\geq k_{0}|\nabla f|^{2}.
\end{equation*}%
From the first two formulas in \eqref{e:3eqs in dim 7} we have%
\begin{eqnarray*}
0 &=&\int_{M}\Big[\frac{3}{16}(\triangle f)^{2}+\frac{3}{4}|(\nabla
^{2}f)_{[-1][a]}|^{2}-\frac{1}{4}|(\nabla ^{2}f)_{[-1][s]}|^{2}-\frac{1}{2}%
T^{0}(\nabla f,\nabla f)-\frac{3}{2}S|\nabla f|^{2}\,\Big]\,\,Vol_{\eta }
\end{eqnarray*}
while the first and the last formula in \eqref{e:3eqs in dim 7} give%
\begin{eqnarray*}
\int_{M}|(\nabla ^{2}f)_{[-1][a]}|^{2}\,Vol_{\eta } &=&\int_{M}\frac{1}{4}%
\sum_{t=1}^{3}\nabla ^{3}f(I_{t}\nabla f,e_{b},I_{t}e_{b})\,Vol_{\eta }.
\end{eqnarray*}
A substitution of the last equation in the previous identity brings us to%
\begin{multline}  \label{mnp}
0 = \\
\int_{M}\Big[-\frac{3}{16}(\triangle f)^{2}-\frac{3}{16}\sum_{t=1}^{3}\nabla
^{3}f(I_{t}\nabla f,e_{b},I_{t}e_{b})+\frac{1}{4}|(\nabla
^{2}f)_{[-1][s]}|^{2}+\frac{1}{2}T^{0}(\nabla f,\nabla f)+\frac{3}{2}%
S|\nabla f|^{2}\,\Big]\,\,\,Vol_{\eta } \\
=\int_{M}\left[\left( \frac{5}{4}T^{0}(\nabla f,\nabla f)+\frac{3}{4}%
S|\nabla f|^{2}-\frac{3}{8}\lambda |\nabla f|^{2}\right) +\frac{1}{4}%
|(\nabla ^{2}f)_{[-1][s]}|^{2}\right]\,Vol_{\eta } \\
+\int_{M}\left [\frac{3}{16}(\triangle f)^{2}-\frac{3}{16}%
\sum_{t=1}^{3}\nabla ^{3}f(I_{t}\nabla f,e_{b},I_{t}e_{b})+\frac{3}{4}%
S|\nabla f|^{2}-\frac34T^{0}(\nabla f,\nabla f)\right]\,Vol_{\eta } \\
=\int_{M}\left[\left( \frac{5}{4}T^{0}(\nabla f,\nabla f)+\frac{3}{4}%
S|\nabla f|^{2}-\frac{3}{8}\lambda |\nabla f|^{2}\right) +\frac{1}{4}%
|(\nabla ^{2}f)_{[-1][s]}|^{2}-\frac {3}{16} P_f(\nabla f)\right]\,Vol_{\eta
},
\end{multline}%
where $P_f$ was defined in \eqref{e:def P}. If we assume that
the $P-$function of $f$ is non-negative,
\begin{equation*}
-\int_M P_f(\nabla f)\vol\geq 0,
\end{equation*}%
then we have%
\begin{equation}\label{eqeq}
0\geq \int_{M}\Big[\left( \frac{1}{8}k_{0}-\frac{3}{8}\lambda |\nabla
f|^{2}\right) +\frac{1}{4}|(\nabla ^{2}f)_{[-1][s]}|^{2}\,\,\Big]\,Vol_{\eta
}.
\end{equation}%
Therefore,
\begin{equation*}
\lambda \geq k_{0}/3.
\end{equation*}
This finishes the proof of Theorem \ref{mainpan}.

\begin{rmrk}
Note that in the extremal case in Theorem~\ref{mainpan} when $\lambda=\frac13k_0$ it follows from \eqref{mnp} that  the corresponding extremal eigenfunction $f$ satisfies the equalities
\begin{equation}\label{extr}(\nabla ^{2}f)_{[-1][s]}=0, \qquad \int_M P_f(\nabla f)\vol=0.
\end{equation}
The first equality in \eqref{extr} together with \eqref{xiii} imply that the horizontal Hessian of an extremal eigenfunction is given by
$$(\nabla ^{2}f)(X,Y)=-\frac{k_0}3fg(X,Y)-\sum_{s=1}^3df(\xi_s)\omega_s(X,Y).
$$
\end{rmrk}

\section{Proof of Corollary~\ref{c:LOin7D}}

In view of Theorem \ref{mainpan} and Proposition \ref{d3n1} we only have to show that if $\lambda=4$ is a an eigenvalue then $M$ is the standard unit 3-Sasakian sphere.
Since $M$ is a qc-Einstein manifold of constant normalized positive scalar curvature $S=2$ then by \cite[Theorem 1.3]{IMV} it follows $M$ is locally 3-Sasakian.  In other words, locally there exists an $SO(3)$-matrix $\Psi$ with smooth entries, such that, the local qc structure determined by $\Psi \cdot \eta$ is 3-Sasakian. The claim then follows as in the higher dimensional case \cite[Theorem 1.2]{IPV1} invoking the (Riemannian) Obata theorem. For this we consider the natural Riemannian metric (also denoted  by) $g$
 defined using the triple of Reeb vector fields  extending $g$ to a metric on $%
M$ by requiring $span\{\xi_1,\xi_2,\xi_3\}=V\perp H \text{ and }
g(\xi_s,\xi_k)=\delta_{sk}. $ The extended metric does not depend on the action of $SO(3)$ on $V$ and the Biquard connection is metric.

On one hand we have the well known fact that a 3-Sasakian manifold of dimension $4n+3$ is Einstein with respect to the extended metric with Riemannian scalar curvature $4n+2$ \cite{Kas}, i.e., the Riemannian Ricci
tensor $Ric^g$ is given by
\begin{equation}  \label{eq21}
Ric^g(A,A)=(4n+2)\,g(A,A).
\end{equation}
By Lichnerowicz' theorem \cite{Li} and \eqref{eq21} we have
\begin{equation}  \label{Lich}
\mu\geq 4n+3=7,
\end{equation}
where $\mu$ is the first eigenvalue of the Riemannian Laplacian of the extended metric.

On the other hand, it follows from the variational characterization of the first eigenvalues and the relation bewteen the Riemannian Laplacian and the sub-Laplacian, see \cite[Proposition~ 5.2]{IPV1} that
\begin{equation}  \label{eigenv}
\mu\leq\lambda+\int_M\sum_{s=1}^3(df(\xi_s))^2 \, Vol_{\eta}
\end{equation}
for any smooth function $f$ with $\int_Mf^2\,\, Vol_{\eta}=1$.

After a possible rescaling of $f$ and using the
divergence formula we have then the following identities
\begin{equation}  \label{eq4}
\begin{aligned} & \lambda=4,\quad \triangle f= 4f,\quad\int_Mf^2\,
Vol_{\eta}=1,& \int_M |\nabla f|^2 \, Vol_{\eta}= 4=\frac
{1}{4}\int_M (\triangle f)^2\, Vol_{\eta}. \end{aligned}
\end{equation}
{  For $\lambda=4$ the first equality in  \eqref{extr}} combined with  the first two equations in \eqref{e:3eqs in dim 7} and
\eqref{xiii} yield
\begin{equation}  \label{eq111}
\int_M\sum_{s=1}^3(df(\xi_s))^2\, Vol_{\eta}=3.
\end{equation}
Now,  from \eqref{eq111} and \eqref{eigenv} we have the inequality
$ \mu\leq 4+3=7 $
which combined with  \eqref{Lich} yields the equality $\mu=4+3=7.$
Therefore, by Obata's result \cite{O3} we conclude that the manifold $(M,g)$ is
isometric to the sphere $S^{7}(1)$ and hence the manifold $(M,g,\mathbb{Q}%
)$ is qc equivalent to the $3$-Sasakian sphere of dimension $7$.
The last statement follows from the fact that any
3-Sasakian manifold satisfies $T^0=U=0, S=2$ \cite{IMV}. This
completes the proof of of Corollary~\ref{c:LOin7D}.

\end{document}